\documentclass{amsart}
\usepackage{latexsym}
\usepackage{amssymb}
\usepackage{enumerate}
\usepackage{amsmath}
\usepackage{amsthm}
\usepackage{xypic}
\usepackage{verbatim}
\usepackage[bookmarks=false]{hyperref}

\def\ints{{\mathbb Z}}
\def\rats{{\mathbb Q}}

\def\reals{{\mathbb R}}

\def\proj{{\mathbb P}}
\def\FF{{\mathbb F}}

\def\Frac{{\text{Frac}}}
\def\Gal{{\text{Gal}}}

\def\ol#1{\overline{#1}}

\newtheorem{theorem}{Theorem}[section]
\newtheorem{prop}[theorem]{Proposition}
\newtheorem{lemma}[theorem]{Lemma}
\newtheorem{corollary}[theorem]{Corollary}
\newtheorem{predefinition}[theorem]{Definition}
\newenvironment{definition}{\begin{predefinition}\rm}{\end{predefinition}}
\newtheorem{preremark}[theorem]{Remark}

\numberwithin{equation}{section}
\begin{document}

\title[Conductors of wild extensions]{Conductors of wild extensions of local fields, especially in mixed characteristic $(0,2)$}

\author{Andrew Obus}
\address{Columbia University, Department of Mathematics, MC4403, 2990 Broadway, New York, NY 10027}
\email{obus@math.columbia.edu}
\thanks{The author was supported by an NSF Postdoctoral 
Research Fellowship in the Mathematical Sciences.  Final preparation of this paper took place at the Max-Planck-Institut f\"{u}r Mathematik in Bonn}

\date{\today}


\begin{abstract}
If $K_0$ is the fraction field of the Witt vectors over an algebraically closed field $k$ of characteristic $p$, we calculate upper bounds on the
conductor of higher ramification for (the Galois closure of) extensions of the form 
$K_0(\zeta_{p^c}, \sqrt[p^c]{a})/K_0$, where $a \in K_0(\zeta_{p^c})$.  Here $\zeta_{p^c}$ is a 
primitive $p^c$th root of unity.  In certain cases, including when $a \in K_0$ and $p=2$, we calculate the conductor exactly.  These calculations can
be used to determine the discriminants of various extensions of $\rats$ obtained by adjoining roots of unity and radicals.
\end{abstract}

\maketitle
The purpose of this paper is to study the higher ramification filtrations of certain wild extensions of discrete valuation fields.  
Let $k$ be an algebraically closed field of characteristic $p > 0$. 
We set $K_0 = \Frac(W(k))$ and $K_c = K_0(\zeta_{p^c})$, where $\zeta_n$ means a primitive $n$th root of unity.  
The main result is the determination
of the higher ramification groups for Galois extensions of the form $K_c(\sqrt[2^c]{a})/K_0$, where $c \geq 1$, $p=2$, and $a \in K_0$ 
(Theorem \ref{Tmain}).  In fact, we do not
explicitly calculate all of the higher ramification groups, but rather the conductor of the extension, which is the highest index for which there exists
a nontrivial higher ramification group for the upper numbering.  Since the subextensions of extensions of this type have a form resembling
that of the original extension, one can 
calculate the conductors of all the subextensions as well.  One can then use properties of the higher ramification groups to show that 
this is enough to calculate all of the higher ramification groups
of $K_c(\sqrt[2^c]{a})/K_0$ (Proposition \ref{Pconductorenough} and the introduction to \S\ref{Schar2}), 
which is in turn enough to calculate the different and discriminant (\cite[IV, Proposition 4 and VI, \S3, Corollary 2]{Se:lf}).  

Additionally, we calculate an upper bound on the conductor of (the Galois closure of)
any extension of the form $K = K_c(\sqrt[p^c]{a})/K_0$, where $p$ is arbitrary and
$a \in K_c$, but not necessarily in $K_0$ (Corollary \ref{C2specific}).  In certain situations, we get an exact value for the conductor 
(Proposition \ref{P2conductor}).  Our calculations in this more general situation are in fact used in the proof of Theorem \ref{Tmain} (in particular, part 
(iig)).  Our techniques are reminiscent of those used by Viviani in \cite{Vi:ra}, where the assumptions
are made that $a \in K_0$ and $p$ is odd.  The main idea is to focus on what we call \emph{$p$-primitive elements} of a mixed characteristic
discrete valuation field (Definition \ref{Dprimitive}).
Extensions obtained by taking roots of such elements are particularly amenable to having their higher ramification groups determined.
We then proceed by writing $K$ as the compositum of extensions coming from roots of $p$-primitive elements and roots of unity, and using
theorems about how higher ramification groups behave under taking the compositum (Lemmas \ref{Lcompositum}, \ref{Lcompositumexact}).

We have two main motivations.  The first comes from \cite{Vi:ra}.  In it, Viviani calculates the higher ramification groups away from $2$ of
all Galois extensions $\rats(\zeta_m, \sqrt[m]{a})/\rats$, 
so long as $m$ is odd and $a \in \rats$ satisfies a technical condition.   He is able to reduce this to
the study of the extensions $\rats_p(\zeta_{p^c}, \sqrt[p^c]{a})/\rats_p$, where $p$ is odd and the $p$-valuation of $a \in \rats_p$ is either prime to $p$ 
or divisible by $p^c$ (hereafter, the ``valuation condition").  
Of course, one can make a base change to the maximal unramified extension $\rats_p^{ur}$ of $\rats_p$ without changing the 
higher ramification groups.  Furthermore, since we are studying algebraic extensions, there is no harm in making a further base change to the 
completion $C$ of
$\rats_p^{ur}$.  We note that, if $k = \ol{\FF_p}$, then $K_0 = C$.  Thus, the calculation of the higher ramification groups in \cite{Vi:ra}
is equivalent to calculating the higher ramification groups of $K_c(\sqrt[p^c]{a})/K_0$ when $p$ is odd, $k = \ol{\FF_p}$, and $a \in K_0$ satisfies the
valuation condition.  Naturally, one would like a similar result when $p=2$, which is what Theorem \ref{Tmain} provides.
Furthermore, we need no valuation condition on $a$ when $p=2$, 
although we are unfortunately not able to eliminate the valuation condition when $p$ is odd.

The second motivation comes from \cite{Ob:fm} and \cite{Ob:fm2}.  
Let $f:Y \to \proj^1$ be a $G$-Galois cover of $\proj^1$ branched at $0$, $1$, and $\infty$, a priori defined over the algebraic closure of $K_0$.
If a $p$-Sylow subgroup of $G$ is of order $p$, then it turns out that $f$ can in fact be defined over a tame extension of $K_0$ (\cite{We:br}).
However, if a $p$-Sylow subgroup of $G$ is cyclic of order $p^r$, 
then the best that can be proven at the moment, especially when $p$ is small, is that $f$ can often be defined over a field of the 
form $K_c(\sqrt[p^c]{a})/K_0$, where $a \in K_c$.  In fact, even the stable model of $f$ can often be defined over such an 
extension.  The bounds that we calculate on the conductors of these extensions (in particular, Corollary \ref{Ccruder}) 
are sufficient to yield aesthetically pleasing statements
of the form ``smaller cyclic $p$-Sylow subgroups lead to smaller conductors of the minimal field of definition over $K_0$") (see \cite[Theorem 1.3]{Ob:fm}
and \cite[Theorem 1.1]{Ob:fm2} for the specific statements).

After some basic results on how higher ramification groups act under compositums and towers of field extensions (\S\ref{Shigher} and \S\ref{Stowers}),
we study the ramification behavior of prime order Kummer extensions and introduce the concept of $p$-primitive elements (\S\ref{Sprime}).  
The technical heart
of the paper is \S\ref{Smain}, where we study the conductor of an extension $K_c(\sqrt[p^c]{a})/K_0$ for $a \in K_0$ by breaking this extension up 
into extensions involving only roots of unity (well understood by \cite{Se:lf}) and prime order Kummer extensions.  We put everything together in 
\S\ref{Schar2} to prove Theorem \ref{Tmain}.

\section*{Acknowledgements}
I thank the referee for useful comments to improve the exposition.  
This work began as an appendix to \cite{Ob:fm}.  I thank the anonymous referee of \cite{Ob:fm} for suggesting that it
be turned into its own paper.  

\section*{Conventions}
Throughout this paper, the valuation $v_K$ on any discrete valuation field $K$ is normalized so that the valuation of a uniformizer is
$1$. 

\section{Higher ramification filtrations}\label{Shigher} 
We state here some facts from \cite[IV]{Se:lf}.
Let $K$ be a complete discrete valuation field with residue field $k$.  If $L/K$ is a finite Galois
extension of fields with Galois group $G$, then the group $G$ has a filtration $G  = G_0
\geq G_i$ ($i \in \reals_{\geq 0}$) for the lower numbering.  If $\pi_L$ is a uniformizer of $L$, this filtration is given by
$$g \in G_i \Leftrightarrow v_L(g\pi_L - \pi_L) \geq i + 1.$$  
There is also a filtration $G = G^0 \geq G^i$ ($i \in \reals \geq 0$) for the upper numbering, defined by 
$G^i = G_{\psi_{L/K}(i)}$, where $\psi_{L/K}: [0, \infty) \to [0, \infty)$ is a certain monotonically increasing, piecewise linear function (\cite[IV, \S3]{Se:lf}).  
The inverse of $\psi_{L/K}$ is denoted $\phi_{L/K}$.  Clearly
$G^{\phi_{L/K}(j)} = G_j$ for $j \in [0, \infty)$.   
The subgroup $G_i$ (resp.\ $G^i$) is known as the \emph{$i$th higher ramification group for
the lower numbering (resp.\ the upper numbering)}.  If $H \leq G$, and $M = L^H$, then it follows from the definitions that 
then the $i$th higher ramification group $H_i$ for the lower numbering for $L/M$ is $G_i \cap H$.  If, furthermore, $H$ is normal, then
the $i$th higher ramification group $(G/H)^i$ for the upper numbering for $M/K$ is $G^i/(G^i \cap H) \leq G/H$ (\cite[IV, Proposition 14]{Se:lf}).
We say that the lower numbering is invariant under subgroups, whereas the upper numbering is invariant under quotients.

The \emph{conductor} of $L/K$, written $h_{L/K}$, is defined by
$$h_{L/K} = \sup_{i\in [0, \infty)}(G^i \neq \{id\})$$ (note that this differs by $1$ from the definition of \cite[p.\ 228]{Se:lf}).
The \emph{highest lower jump} of $L/K$, denoted $\ell_{L/K}$, is defined by
$$\ell_{L/K} = \sup_{i\in [0, \infty)}(G_i \neq \{id\}).$$  Of course, $\psi_{L/K}(h_{L/K}) = \ell_{L/K}$ and $\phi_{L/K}(\ell_{L/K}) = h_{L/K}$.  

The following lemma is easy (for a proof, see e.g. \cite[Lemma 2.3]{Ob:fm}).
\begin{lemma}\label{Lcompositum}
Given $L/K$ as above, let $M_1, \ldots, M_{\ell}$ be subfields of $L$ containing $K$ whose compositum is $L$.
Then $h_{L/K} = \max_i(h_{M_i/K})$.
\end{lemma}

\begin{lemma}\label{Lcompositumexact}
Given $L/K$ as above, let $M_1$, $M_2$, and $M_3$ be subfields of $L$ containing $K$, the compositum of any two of which is $L$.  If
$h_{M_1/K} > h_{M_2/K}$, then $h_{M_3/K} = h_{M_1/K} = h_{L/K}$.
\end{lemma}

\begin{proof}
By Lemma \ref{Lcompositum} applied to $M_1$ and $M_2$, we have $h_{L/K} = h_{M_1/K}$.  Then the same lemma, applied to
$M_2$ and $M_3$, implies that $h_{L/K} = h_{M_3/K}$.
\end{proof}

\begin{prop}\label{Pconductorenough}
Given $L/K$ a $G$-Galois extension as above, 
the higher ramification filtration is completely determined by knowing the conductor of each Galois extension $M/K$, where 
$K \subseteq M \subseteq L$.
\end{prop}

\begin{proof}
Clearly it is enough to determine $G^i$ for all $i \geq 0$.  For any normal subgroup $H \leq G$, 
invariance under quotients shows that $H \geq G^i$ iff $(G/H)^i = \{id\}$, that is, if $h_{L^H/K} < i$ (the ramification filtration on $G/H$ 
corresponds to the extension $L^H/K$).  
If $H_1$ and $H_2$ are two normal subgroups of $G$, then $L^{H_1 \cap H_2}$ is the compositum of
$L^{H_1}$ and $L^{H_2}$.  Thus Lemma \ref{Lcompositum} shows that there is a unique minimal normal subgroup $H \leq G$ 
such that $h_{L^H/K} < i$.  Since $G^i$ is normal in $G$ (\cite[IV, Proposition 1]{Se:lf}), we conclude that $G^i$ is, in fact, 
this minimal normal subgroup $H$.
\end{proof}

The following result of Serre is proved only in the case $K_0 = \rats_p$, but the proof immediately extends to the context below.

\begin{prop}[\cite{Se:lf}, IV, Proposition 18]\label{Pcyclic}
Let $k$ be an algebraically closed field of characteristic $p$, let $K_0 = \Frac(W(k))$, and for $n \geq 1$, let $K_n = K_0(\zeta_{p^n})$.  
Identify $G := \Gal(K_n/K_0)$ with $(\ints/p^n)^{\times}$. 
Then $G_0 = G$, and for $0 < i \leq p^{n-1}$, we have 
$$G_i = \{a \in G \, | \, a \equiv 1 \pmod{p^{\lfloor \log_p i \rfloor +1}} \}.$$  
In particular, one calculates that $h_{K_n/K_0} = n-1$.
\end{prop}

\section{Ramification filtrations in towers}\label{Stowers}
In this section, we give several results about how conductors act in towers.


\begin{lemma}\label{Lramintowers}
Let $K$ be a complete discrete valuation field with
algebraically closed residue field.  Let $K \subseteq L \subseteq M$ be finite Galois extensions such that $M$ is Galois over $K$.  
Then, either $h_{M/K} = h_{L/K}$ (in which case $h_{M/L} \leq \ell_{L/K}$), 
or $h_{M/K} > h_{L/K}$, in which case $$\frac{1}{[L:K]}(h_{M/L} - \ell_{L/K}) = h_{M/K} - h_{L/K}.$$
\end{lemma}

\begin{proof}
Let $G = \Gal(M/K), H = \Gal(M/L),$ and $G/H = \Gal(L/K)$.
Since the upper numbering is invariant under taking quotients, $h_{M/K} \geq h_{L/K}$.  
Let $j$ be the greatest index such that $G_j \gneq H_j$, i.e., $j$ is the greatest index such that
$G_j \not \leq H$.  Then $G_i = H_i$ for all $i > j$.
Now, $G_j = G^{\phi_{M/K}(j)}$.  By invariance
under quotients, we have that 
\begin{equation}\label{Ehlk}
\phi_{M/K}(j) = h_{L/K}.
\end{equation}  
By applying $\psi_{L/K}$ to both
sides of (\ref{Ehlk}), and using the fact that $\phi_{M/K} = \phi_{L/K} \circ \phi_{M/L}$ (\cite[IV, Proposition 15]{Se:lf}), we obtain that 
\begin{equation}\label{Ellk}
\phi_{M/L}(j) = \ell_{L/K}.
\end{equation}

Suppose that $h_{M/K} = h_{L/K}$.  Then, applying $\psi_{L/K}$, we get  $\psi_{L/K}(h_{M/K}) = \ell_{L/K}$.  But 
\begin{eqnarray*}
\psi_{L/K}(h_{M/K}) &=& (\psi_{L/K} \circ \phi_{M/K})(\ell_{M/K}) = (\psi_{L/K} \circ \phi_{L/K} \circ \phi_{M/L})(\ell_{M/K}) \\
&=& \phi_{M/L}(\ell_{M/K}) \geq \phi_{M/L}(\ell_{M/L}) = h_{M/L},
\end{eqnarray*} 
so $h_{M/L} \leq \ell_{L/K}$.

Now suppose that $h_{M/K} > h_{L/K}$. Applying $\psi_{M/K}$ to both sides yields $$\ell_{M/K} > \psi_{M/K}(h_{L/K}) = j,$$ the last equality following from (\ref{Ehlk}).  
Since $G_i = H_i$ for $i > j$, we have that 
\begin{equation}\label{Ekandl}
\ell_{M/L} = \ell_{M/K}.
\end{equation}
It also follows that, for all $i > j$, we have $[G:G_i] = [L:K][H:H_i]$, thus the slope of $\phi_{M/L}$ is $[L:K]$ times the slope of $\phi_{M/K}$ for arguments greater than $j$.  
So
\begin{equation}\label{Efinal}
\frac{1}{[L:K]}(\phi_{M/L}(\ell_{M/K}) - \phi_{M/L}(j)) = (\phi_{M/K}(\ell_{M/K}) - \phi_{M/K}(j)).
\end{equation}  
By (\ref{Ekandl}), 
$\phi_{M/L}(\ell_{M/K}) = \phi_{M/L}(\ell_{M/L}) = h_{M/L}$, and $\phi_{M/K}(\ell_{M/K}) = h_{M/K}$.   
Now substituting (\ref{Ehlk}) and (\ref{Ellk}) into (\ref{Efinal}) gives the statement of the lemma.
\end{proof}

\begin{corollary}\label{C2goingdown} 
Let $K$ be a complete discrete valuation field with algebraically closed residue field, and let $K \subset L \subseteq M$ be 
Galois field extensions so that $M$ is
Galois over $K$.  Assume that $[L:K] = p$.  Then 
$$h_{M/K} = \max(h_{L/K}, \frac{p-1}{p}h_{L/K} + \frac{1}{p}h_{M/L}).$$
\end{corollary}

\begin{proof}
Since $h_{L/K}$ is equal to the (only) lower jump $\ell_{L/K}$ of $L/K$, Lemma \ref{Lramintowers} implies that either $h_{M/K} =
h_{L/K}$ or $\frac{1}{p}(h_{M/L} - h_{L/K}) = h_{M/K} - h_{L/K}$.  Solving for $h_{M/K}$ proves the corollary.
\end{proof}

The next corollary generalizes \cite[Lemme 1.1.4]{Ra:sp}.

\begin{corollary}\label{C2compositum}
Let $K$ be a complete discrete valuation field with algebraically closed residue field, let $K'$ be a $\ints/p$-extension of $K$, and let $L$ be any finite Galois
extension of $K$.  Write $L'$ for the compositum of $L$ and $K'$ in some algebraic closure of $K$ (we do not assume the extensions 
are linearly disjoint).  Then, if $h_{L/K} > h_{K'/K}$, we have $h_{L'/K'} = ph_{L/K} - (p-1)h_{K'/K}$.  If $h_{L/K} \leq h_{K'/K}$, 
then $h_{L'/K'} \leq h_{K'/K}$.
\end{corollary}

\begin{proof}
We draw a diagram of the situation as follows:

\[
\xymatrix{
L \ar@{^{(}->}[r]^-{h_{L'/L}} & L'  
\\
K  \ar@{^{(}->}[r]^{h_{K'/K}}_{\ints/p}  \ar@{^{(}->}[u]^{h_{L/K}} & K'  
\ar@{^{(}->}[u]_{h_{L'/K'}} }
\]

If $h_{L/K} > h_{K'/K}$, then by Lemma \ref{Lcompositum}, $h_{L'/K} = h_{L/K} > h_{K'/K}$.
In this situation the corollary follows from Lemma \ref{Lramintowers} applied to the tower $K \subseteq K' \subseteq L'$, using the fact that 
$\ell_{K'/K} = h_{K'/K}.$  
If $h_{L/K} \leq h_{K'/K}$, then, by Lemma \ref{Lcompositum}, we have $h_{L'/K} = h_{K'/K}$.  Then $h_{L'/K'} \leq \ell_{K'/K} = h_{K'/K}$ by Lemma \ref{Lramintowers} 
applied to the tower $K \subseteq K' \subseteq L'$.
\end{proof}

Recall that $K_0$ is $\Frac(W(k))$, where $k$ is an algebraically closed field of characteristic $p$, and that $K_r = K_0(\zeta_{p^r})$.

\begin{corollary}\label{C2rootsofunity}
Let $r \geq 1$, and let $M/K_{\ell}$ be a finite extension such that $M/K_0$ is Galois.  Then
$$h_{M/K_0} = \max\left(\ell - 1, \ell-\frac{p}{p-1} + \frac{1}{(p-1)p^{\ell-1}}\left(h_{M/K_{\ell}} + 1\right)\right).$$
\end{corollary} 

\begin{proof}  
By Proposition \ref{Pcyclic}, the conductor of $K_{\ell}/K_0$ is $\ell-1$, and the greatest
lower jump is $p^{\ell-1} - 1$.  So by Lemma \ref{Lramintowers}, applied to the tower $K_0 \subset K_{\ell} \subset M$, either $h_{M/K_0} = \ell-1$ or
$\frac{1}{(p-1)p^{\ell-1}}(h_{M/K_{\ell}} - (p^{\ell-1} - 1)) = h_{M/K_0} - (\ell-1)$.  Solving for $h_{M/K_0}$ yields the corollary.
\end{proof}

\section{Prime order extensions in mixed characteristic}\label{Sprime}

If $K$ is a mixed characteristic $(0,p)$ discrete valuation ring, then we write $e_K = v_K(p)$, which is the absolute ramification index of $K$.
 
\begin{lemma}\label{Lpthpowers}
Let $K$ be a finite extension of $K_0$, and suppose $a = 1 + t \in K$ with $v_K(t) < \frac{p}{p-1}e_K$.  
Then any $p$th root of $a$ can be expressed (in an appropriate 
finite extension $L/K$) as $1 + r$, where $v_L(r) =  \frac{v_L(t)}{p}$.
\end{lemma}

\begin{proof}
If $r \in L$ is as in the lemma, then $1+ t = (1+r)^p = 1 + \sum_{i=1}^{p} \binom{p}{i}r^i$.
Then $v_L(t) \geq \min_{1 \leq i \leq p}(v_L\left(\binom{p}{i}r^i\right))$.  Since $v_L(t) < \frac{p}{p-1}e_L$, we have $v_L(r) < \frac{1}{p-1}e_L$, which in 
turn implies that
the minimum is realized for $i = p$.  So $v_L(t) = v_L(r^p)$, from which the lemma follows.
\end{proof}

\begin{definition}\label{Dprimitive}
Let $K/K_1$ be finite.  If $a = 1 + t$ is an element of $K$ such that either $p \nmid v_K(a)$, or $0 < v_K(t) < \frac{p}{p-1}e_K$ and $p \nmid v_K(t)$, 
then we will say that $a$ is \emph{$p$-primitive for $K$}.
\end{definition}

\begin{lemma}\label{L2extensions}
Let $K$ be a finite extension of $K_1$, and let $L/K$ be a (nontrivial) $\ints/p$-extension. 
\begin{enumerate}[(i)]
\item We have $L = K(\sqrt[p]{a})$, where $a$ is $p$-primitive for $K$. 
\item For $a$ and $t$ as in (i), the conductor $h_{L/K}$ is $\frac{p}{p-1}e_K - v_K(t)$.
\end{enumerate}
\end{lemma}

\begin{proof}
\emph{To (i):}
By Kummer theory, we know we can find $a$ such that $L \cong K(\sqrt[p]{a})$.  In choosing $a$, we are free to multiply by
elements of $(K^{\times})^p$.  Thus we can assume that $0 \leq v_K(a) < p$.  If $v_K(a) > 0$, then $a$ is $p$-primitive
for $K$.  If $v_K(a) = 0$, we can use the fact that 
$k$ is algebraically closed to multiply $a$ by a $p$th power so that it is congruent to $1$ modulo a uniformizer.  So write $a = 1 + t$, where 
$v_K(t) > 0$.  If $v_K(t) \geq \frac{p}{p-1}e_K$, then $a$ is a $p$th power
in $K$ (\cite[\S 0.3]{Ep:wr}), contradicting the nontriviality of $L/K$.  So $v_K(t) < \frac{p}{p-1}e_K$.  If $p \nmid v_K(t)$, then $a$ is 
$p$-primitive for $K$.  If $p | v_K(t)$,  then write $t = (\pi_K)^{p\nu}w$, where $\pi_K$ is a uniformizer of $K$ and $v_K(w) = 0$.  Let $y \in K$ be 
such that $v_K(y^p + w) > 0$ (we can find such a $y$
because $k$ is algebraically closed).  If $a' = a(1 + (\pi_K)^{\nu}y)^p$, then $a' = 1 + t'$ with
$v_K(t') > v_K(t)$.  So replace $a$ by $a'$ and $t$ by $t'$ and repeat until $a$ is $p$-primitive for $K$.
This process must terminate eventually, as the valuation of $t$ is bounded by $\frac{p}{p-1}e_K$.
\\
\\
\emph{To (ii):} (cf.\ \cite[Theorems 5.6 and 6.3]{Vi:ra})
We first calculate a uniformizer $\pi_L$ of $L$.
Then, if $\sigma$ is a generator of $\Gal(L/K)$, we will determine $v(\sigma(\pi_L - \pi_L))$.  Let 
$\sqrt[p]{a}$ be a choice of $p$th root such that $\sigma(\sqrt[p]{a}) = \zeta_p \sqrt[p]{a}$.

Suppose $p \nmid v_K(a)$.  Choose integers $m$ and $n$ such that $mp + nv_K(a) = 1$.  
Thus $v_L(\pi_K^m a^{n/p}) = 1$, so we set $\pi_L = \pi_K^m a^{n/p}$.  
Since $v_L(\zeta_p-1) = \frac{1}{p-1}e_L$, we have $v_L(\sigma(\pi_L) - \pi_L) = mp + v_L(a^{n/p}) + \frac{1}{p-1}e_L$.
Since $L/K$ is a $\ints/p$-extension, the definition of the conductor gives
$h_{L/K} = mp + \frac{n}{p}v_L(a) + \frac{1}{p-1}e_L - 1 = \frac{p}{p-1}e_K$.  Since $v_K(t) = 0$, this proves the 
proposition in this case.

Now, suppose $p \nmid v_K(t) > 0$.  By Lemma \ref{Lpthpowers}, $v_L(\sqrt[p]{a} - 1) = \frac{v_L(t)}{p} = v_K(t)$.
Since $p \nmid v_K(t)$, there exist $m, n \in \ints$ such that
$mp + n v_K(t) = 1$.  We can even require $0 < n < p$.
Then $\pi_L := \pi_K^m(\sqrt[p]{a} - 1)^n$ is a uniformizer of $L$.

Computing, we find
\begin{eqnarray*}
\sigma(\pi_L) - \pi_L &=& \pi_K^m\left((\sqrt[p]{a} - 1 + (\zeta_p-1)\sqrt[p]{a})^n - (\sqrt[p]{a}-1)^n\right) \\
&=& \pi_K^m \left( \sum_{i=0}^{n-1} \binom{n}{i} (\sqrt[p]{a} - 1)^i((\zeta_p-1)\sqrt[p]{a})^{n-i} \right)
\end{eqnarray*} 
By assumption, $v_L(\zeta_p-1) = \frac{p}{p-1}e_K > v_K(t) = \frac{v_L(t)}{p} = v_L(\sqrt[p]{a} - 1)$.
Also, since $n < p$, it follows that $\binom{n}{i}$ has valuation $0$.  Thus all terms in the sum above have different valuations,
and the term of lowest valuation corresponds to $i = n-1$.
Applying $v_L$ to this term gives $mp + (n-1)v_K(t) + \frac{1}{p-1}e_L = 1 + \frac{p}{p-1}e_K - v_K(t)$.
Since $L/K$ is a $\ints/p$-extension, the definition of the conductor gives $h_{L/K} = \frac{p}{p-1}e_K - v_K(t)$.  
\end{proof}

\begin{lemma}\label{R2extensions}
If $L \cong K(\sqrt[p]{a})$ is a $\ints/p$-extension of $K$
with $a = 1 + t$, then $h_{L/K} \leq \frac{p}{p-1}e_K - v_K(t)$, with equality holding iff $a$ is $p$-primitive for $K$.
In particular, if $a' = 1+t'$ is $p$-primitive for $K$ and $a/a'$ is a $p$th power in $K$, then $v_K(t') \geq v_K(t)$.
\end{lemma}

\begin{proof}
The proof of Lemma \ref{L2extensions} shows that if $a = 1 + t$ is not $p$-primitive for $K$, then one can find $a'$ $p$-primitive for $K$ such 
that $a' = 1 + t'$ with $v_K(t') > v_K(t)$, and $\frac{a}{a'}$ is a $p$th power in $K$.  The lemma then follows from Lemma \ref{L2extensions}.
\end{proof}

\begin{corollary}\label{C2extensions}
Let $L_0$ be a finite extension of $K_1$.  Let $a = 1+t$ be
$p$-primitive for $L_0$.  For $i > 0$, write $L_i = L_0(\sqrt[p^i]{a})$.
Then $[L_c:L_0] = p^c$ for $c > 0$, and for $0 < i \leq c$, the conductor 
of $L_i/L_{i-1}$ is $\frac{p^i}{p-1} e_{L_0} - v_{L_0}(t)$.
\end{corollary}

\begin{proof}
If $p \nmid v(a)$, then clearly $[L_c:L_0] = p^c$ and $v_{L_0}(t) = 0$.  Also, $p \nmid v_{L_{i-1}}(\sqrt[p^{i-1}]{a})$. 
Then Lemma \ref{L2extensions} shows that $h_{L_i/L_{i-1}} = \frac{p}{p-1} e_{L_{i-1}} = \frac{p^i}{p-1} e_{L_0}.$

Suppose $v(a) = 0$ and $p \nmid v(t) > 0$.  Since $L_0$ contains the $p$th roots of unity, we know that $L_1/L_0$ is Galois.
So we apply Lemma \ref{L2extensions} to see that $[L_1:L_0] = p$ and $h_{L_1/L_0}$ is as desired.
Choose a $p$th root $\sqrt[p]{a}$.  By Lemma \ref{Lpthpowers}, we have that $v_{L_1}(\sqrt[p]{a} - 1) = \frac{v_{L_1}(t)}{p} = v_{L_0}(t)$, which is prime 
to $p$.  So $\sqrt[p]{a}$ is $p$-primitive for $L_1$.  Applying 
Lemma \ref{L2extensions} to $L_2/L_1$ and $\sqrt[p]{a}$ shows that $[L_2:L_1] = p$ and $$h_{L_2/L_1} = \frac{p}{p-1}e_{L_1} - v_{L_1}(\sqrt[p]{a} - 1) 
= \frac{p^2}{p-1} e_{L_0} - v_{L_0}(t),$$ as desired.  Repeating this process up to reaching $L_c$ yields the corollary.
Note that in each case, $L_i/L_{i-1}$ is Galois, 
\end{proof}

\section{Conductors of a certain class of metabelian extensions}\label{Smain}

Recall that $K_r = K_0(\zeta_{p^r})$. 
Write $v_r$ for the normalized valuation on $K_r$ such that a uniformizer has valuation $1$, that is, $v_r = v_{K_r}$. 

\begin{lemma}\label{L2conductor}
Choose integers $\ell$ and $c$ such that $1 \leq \ell \leq c$.  Let $a  = 1+t \in K_{\ell}$ such that $a$ is
$p$-primitive for $K_{\ell}$.  
Then the conductor of $K_c(\sqrt[p^c]{a})$ over $K_{\ell}(\sqrt[p^c]{a})$ is less than or equal to
$p^{c + \ell - 1} - v_{\ell}(t),$ which is the conductor of $K_{\ell}(\sqrt[p^c]{a})$ over $K_{\ell}(\sqrt[p^{c-1}]{a})$.
\end{lemma}

\begin{proof}
If $c = \ell$, the lemma follows from Corollary \ref{C2extensions}, so assume $c > \ell$.
For each $i$, $0 \leq i \leq c$, let $h_i$ be the conductor of $K_c(\sqrt[p^i]{a})$ over $K_{\ell}(\sqrt[p^i]{a})$.  Then, using 
Proposition \ref{Pcyclic}, one calculates $h_0 = p^{\ell-1}((c- \ell)(p-1) + 1) - 1$. 
Furthermore, let $a_i$ be the conductor of $K_{\ell}(\sqrt[p^i]{a})$ over $K_{\ell}(\sqrt[p^{i-1}]{a})$.  
By Corollary \ref{C2extensions}, we have $a_i = p^{i + \ell - 1} - v_{\ell}(t).$  Note that $v_{\ell}(t) < \frac{p}{p-1}e_{K_{\ell}} = 
p^{\ell}$.  We must show that $h_c \leq a_c$.
Our diagram of field extensions and conductors looks like this:

\[
\xymatrix{
K_{c} \ar@{^{(}->}[r] & K_{c}(\sqrt[p]{a})  \ar@{^{(}->}[r] & \cdots  \ar@{^{(}->}[r] & 
K_{c}(\sqrt[p^{c-1}]{a})  \ar@{^{(}->}[r] & K_{c}(\sqrt[p^c]{a}) 
\\
K_{\ell}  \ar@{^{(}->}[r]^-{a_1}  \ar@{^{(}->}[u]^{h_0} & K_{\ell}(\sqrt[p]{a})  \ar@{^{(}->}[r]^-{a_2}
\ar@{^{(}->}[u]^{h_1} & 
\cdots \ar@{^{(}->}[r]^-{a_{c-1}} & K_{\ell}(\sqrt[p^{c-1}]{a})  \ar@{^{(}->}[r]^-{a_c} \ar@{^{(}->}[u]^{h_{c-1}}  & 
K_{\ell}(\sqrt[p^c]{a}).  \ar@{^{(}->}[u]^{h_c}}
\]

If there exists an $i$, $0 \leq i < c$, such that $h_i \leq a_{i+1}$, then repeated application of Corollary \ref{C2compositum} shows that
$h_c \leq  a_c$.  So assume otherwise.  Then we have the chain of (in)equalities below 
(the first comes from repeated application of Corollary \ref{C2compositum}):
\begin{eqnarray*}
h_c &=& p^ch_0 - p^{c-1}(p-1)a_1 - p^{c-2}(p-1)a_2 - \cdots - p(p-1)a_{c-1} - (p-1)a_c \\
&=& p^ch_0 - c(p-1)(p^{c + \ell - 1}) + (p^c-1)v_{\ell}(t) \\
&=& p^{c + \ell - 1}(1 - \ell(p-1)) + (p^c - 1)v_{\ell}(t) - p^c \\
&=& a_c - \ell(p-1)(p^{c + \ell -1}) + p^cv_{\ell}(t) - p^c \\
&<& a_c - \ell(p-1)(p^{c + \ell -1}) + p^{c+\ell} - p^c \\
&=& a_c + p^{c+\ell}(1 - \frac{\ell(p-1)}{p} - \frac{1}{p^{\ell}}) \\
&\leq& a_c. 
\end{eqnarray*}
\end{proof}

\begin{prop}\label{P2conductor}
Choose integers $\ell \geq 1$ and $c \geq 1$.  Let $a  = 1+t \in K_{\ell}$ such that $a$ is
$p$-primitive for $K_{\ell}$.  Write $K = K_{\max(\ell, c)}(\sqrt[p^c]{a})$.  
Then $K$ is Galois over $K_{\ell}$.  Write $L$ for the Galois closure of $K$ over $K_0$.
\begin{enumerate}[(i)]
\item The conductor of $K/K_{\ell}$ is $(c(p-1)+1)p^{\ell-1} - v_{\ell}(t)$.
\item The conductor of $L/K_0$ is $\max(\ell - 1, c + \ell - 1 - \frac{v_{\ell}(t) - 1}{p^{\ell -1}(p-1)}).$
\end{enumerate}
\end{prop}	

\begin{proof}
The extension $K/K_{\ell}$ is clearly Galois.  Let the $a_i$ be defined as in the proof of Lemma \ref{L2conductor}.  Then 
as in that proof, we have $a_i = p^{i + \ell + 1} - v_{\ell}(t)$.
\\
\\
\emph{To (i):}
If $c \geq \ell$, then 
by Corollary \ref{C2goingdown} applied to $K_{\ell}(\sqrt[p^{c-1}]{a}) \subseteq K_{\ell}(\sqrt[p^{c}]{a}) \subseteq 
K$, and using Lemma \ref{L2conductor}, we have that the conductor of $K$ over $K_{\ell}(\sqrt[p^{c-1}]{a})$ is 
$a_c$.  If $c < \ell$, then the conductor of $K$ over $K_{\ell}(\sqrt[p^{c-1}]{a})$ is $a_c$ by definition.
In both cases, applying Corollary \ref{C2goingdown} repeatedly to the extensions
$K_{\ell}(\sqrt[p^{i-1}]{a}) \subseteq K_{\ell}(\sqrt[p^{i}]{a}) \subseteq K$ as $i$ ranges from 
$c-1$ to $1$, we obtain that the conductor of $K/K_{\ell}$ is 
$$\frac{1}{p^{c-1}}a_c + \sum_{i=1}^{c-1}\frac{p-1}{p^i}a_i.$$  This is equal to $(c(p-1)+1)p^{\ell-1} - v_{\ell}(t)$. 
\\
\\
\emph{To (ii)}  Since $K$ is Galois over $K_{\ell}$, its Galois closure $L$ over $K_0$ is a compositum of conjugate 
extensions, each with the same conductor over $K_{\ell}$.  By Lemma \ref{Lcompositum} and part (i) of this proposition, 
$h_{L/K_{\ell}} =  h_{K/K_{\ell}} = (c(p-1)+1)p^{\ell-1} - v_{\ell}(t)$.  By Corollary \ref{C2rootsofunity}, we obtain
$$h_{L/K_0} = \max\left(\ell - 1, c + \ell - 1 - \frac{v_{\ell}(t) - 1}{p^{\ell-1}(p-1)}\right).$$
\end{proof}

\begin{corollary}\label{C2specific}
Choose integers $\ell \geq 1$ and $c \geq 1$.  Let $\alpha \in K_{\ell}$, not necessarily p-primitive for $K_{\ell}$, with $v_{\ell}(\alpha) \geq 0$.
Write $L$ for the Galois closure of $K := K_{\max(\ell, c)}(\sqrt[p^c]{\alpha})$ over $K_0$.    
We know $\alpha = \alpha'\beta^p$, with either $\alpha' = 1$ or $\alpha'$ $p$-primitive for $K_{\ell}$, and $\beta \in K_{\ell}$ with $v_{\ell}(\beta) \geq 0$.  
Write $\alpha' = 1 + t_{\alpha'}$ and $\beta = 1 + t_{\beta}$.  
Then $$h_{L/K_0} \leq \mu := \max\left(\max(\ell, c) - 1, c + \ell - 1 - \frac{v_{\ell}(t_{\alpha'}) - 1}{p^{\ell -1}(p-1)},
c + \ell - 2 - \frac{v_{\ell}(t_{\beta}) - 1}{p^{\ell -1}(p-1)}\right).$$
\end{corollary}

\begin{proof}
 It is clear that $L$ can be embedded into the Galois closure of $K_{\max(\ell, c)}(\sqrt[p^c]{\alpha'}, 
\sqrt[p^{c-1}]{\beta})$ over $K_0$.
By Lemma \ref{Lcompositum}, it suffices to show that the Galois closures $L'$ of $K_{\max(\ell, c)}(\sqrt[p^c]{\alpha'})$ and
$L''$ of $K_{\max(\ell, c-1)}(\sqrt[p^{c-1}]{\beta})$ over $K_0$ satisfy $h_{L'/K_0} \leq \mu$ and $h_{L''/K_0} \leq \mu$.  By 
Proposition \ref{P2conductor} (ii), $h_{L'/K_0} = \max\left(\ell - 1, c + \ell - 1 - \frac{v_{\ell}(t_{\alpha'}) - 1}{p^{\ell -1}(p-1)}\right) \leq \mu$
when $\alpha'$ is $p$-primitive for $K_{\ell}$.  By Proposition \ref{Pcyclic}, $h_{L'/K_0} = \max(\ell, c) - 1 \leq \mu$ 
when $\alpha' = 1$.

If $c=1$ we are done, so assume $c \geq 2$.  Pick $\beta'$, $\gamma \in K_{\ell}$ such that $\beta = \beta'\gamma^p$ and either 
$\beta' = 1$ or $\beta'$ is
$p$-primitive for $K_{\ell}$.  Then $L''$ can be embedded into the compositum of the Galois closures
$M$ of $K_{\max(\ell, c-1)}(\sqrt[p^{c-1}]{\beta'})$ and $M'$ of $K_{\max(\ell, c-2)}(\sqrt[p^{c-2}]{\gamma})$ over $K_0$.  If $\beta' = 1$, then
$h_{M/K_0} = \max(\ell-1, c-2) \leq \mu$.  
If $\beta'$ is $p$-primitive for $K_{\ell}$, then $\beta' = 1 + t_{\beta'}$ with
$v_{\ell}(t_{\beta'}) \geq v_{\ell}(t_{\beta})$.  We then have, by Proposition \ref{P2conductor}, that  
$h_{M/K_0} = \max(\ell - 1, c + \ell - 2 - \frac{v_{\ell}(t_{\beta'}) - 1}{p^{\ell-1}(p-1)}) \leq \mu$.

If $c=2$ we are done, so assume $c \geq 3$.
By Lemma \ref{Lcompositum}, it remains to prove that $h_{M'/K_0} \leq \mu$.  We prove by induction on $c$ that 
$$h_{M'/K_0} \leq  \max(\ell-1, c+\ell - 3 + \frac{1}{p^{\ell-1}(p-1)}),$$ which is less than $\mu$ because $v_{\ell}(t_{\alpha'}) < p^{\ell}$.  
Write $\gamma = \gamma'\delta^p$, where either $\gamma'=1$ or $\gamma'$ is $p$-primitive for $K_{\ell}$.
If $c=3$, then $M'$ is the Galois closure of  $K_{\max(\ell, c-2)}(\sqrt[p^{c-2}]{\gamma'})$ over $K_0$, 
and we conclude by Proposition \ref{P2conductor}.
If $c > 3$, then $M'$ is contained in the compositum of the Galois closures $N$ of $K_{\max(\ell, c-2)}(\sqrt[p^{c-2}]{\gamma'})$ and 
$N'$ of $K_{\max(\ell, c-3)}(\sqrt[p^{c-3}]{\delta})$ over $K_0$.  
By Proposition \ref{P2conductor}, $h_{N/K_0} \leq \max(\ell-1, c + \ell - 3 + \frac{1}{p^{\ell-1}(p-1)})$,
and by the induction hypothesis, the same holds for $h_{N'/K_0}$.  We conclude using Lemma \ref{Lcompositum}.
\end{proof}

The following version of Corollary \ref{C2specific} will be useful in \S\ref{Schar2} and \cite{Ob:fm}.

\begin{corollary}\label{Ccruder}
Choose integers $\ell \geq 1$ and $c \geq 1$.  Let $\alpha \in K_{\ell}$, not necessarily p-primitive for $K_{\ell}$, with $v_{\ell}(\alpha) \geq 0$.
Write $L$ for the Galois closure of $K := K_d(\sqrt[p^c]{\alpha})$ over $K_0$, where $d \geq \max(\ell, c)$.    
Write $\alpha = 1 + t_{\alpha}$.  
Then $$h_{L/K_0} \leq \mu := \max\left(d-1, c + \ell - 1 - \frac{v_{\ell}(t_{\alpha}) - 1}{p^{\ell -1}(p-1)},
c + \ell - 2 + \frac{1}{p^{\ell -1}(p-1)}\right).$$
\end{corollary}

\begin{proof}
By Lemma \ref{R2extensions}, we have $v_{\ell}(t_{\alpha'}) \geq v_{\ell}(t_{\alpha})$, where $t_{\alpha'}$ is from Corollary \ref{C2specific}.  Furthermore,
$v_{\ell}(t_{\beta}) \geq 0$, where $t_{\beta}$ is from Corollary \ref{C2specific}.  So our corollary follows form Corollary \ref{C2specific}, along with 
Lemma \ref{Lcompositum} and the fact that $h_{K_d/K_0} = d-1$ (Proposition \ref{Pcyclic}).
\end{proof}

\section{Extensions of the form $K_c(\sqrt[2^c]{a})$}\label{Schar2}
In this section, we assume $p=2$.  In order to understand the higher ramification groups above $2$ in an extension 
$\rats(\zeta_{2^c}, \sqrt[2^c]{a})/\rats$, when $a \in \rats$, it suffices, as mentioned in the introduction, 
to make a base change to the completion of the maximal unramified extension
of $\rats_2$ (this is $K_0$, when $k = \ol{\FF_2}$).  We work in the more general context of an extension 
$K = K_c(\sqrt[2^c]{a})/K_0$, with $a \in K_0$.  Note that, since $p=2$, we have $K_0 = K_1$, so the results of \S\ref{Smain} apply.

By Proposition \ref{Pconductorenough}, in order to determine the higher ramification filtration of the Galois extension $K/K_1$, we need only 
determine the conductor of
each Galois subextension $L$ of $K/K_1$.  Each such subextension can be written in the form $L = K_{c'}(\sqrt[2^{c''}]{a'})$, 
where $c' \geq c''$ and $a' \in K_1$.  
Then Lemma \ref{Lcompositum} implies that $h_{L/K_1} = \max(c'-1, h_{L'/K_1})$, where $L' = K_{c''}(\sqrt[2^{c''}]{a'})$. Since $L'/K_1$ is in 
the same form as our original extension $K/K_1$, 
we content ourselves with finding the conductor $h_{K/K_1}$.  For more details on the structure of subextensions of $K/K_1$, see e.g.
\cite{JV:gg} and \cite{dOV:ls}.

After multiplying $a$ by an element of $(K_1^{\times})^{2^c}$, which does not change the extension, we may assume that
$0 \leq v_{K_1}(a) < 2^c$.  Write $a = 2^nb$, where $v_{K_1}(b) = 0$ and $0 \leq n < 2^c$.
After multiplying again by an element of $(K_1^{\times})^{2^c}$, we may assume that $b \equiv 1 \pmod{2}$.

\begin{theorem}\label{Tmain}
Let $K = K_c(\sqrt[2^c]{a})$, where $a = 2^nb \in K_0 = K_1$.  Assume that $0 \leq n < 2^c$ and $b \equiv 1 \pmod{2}$. 
\begin{enumerate}[(i)]
\item If $n$ is odd, then $h_{K/K_1} = c+1$.
\item Suppose $2|n$.

\begin{enumerate}[(a)]
\item If $c=1$ and $b \equiv 1 \pmod{4}$, then $K=K_1$.
\item If $c=1$ and $b \equiv 3 \pmod{4}$, then $h_{K/K_1} = 1$.  
\item If $4|n$, $c > 1$, and $b \equiv 1 \pmod{4}$, then $h_{K/K_1} = c-1$.
\item If $4|n$, $c > 1$, and $b \equiv 3 \pmod{4}$, then $h_{K/K_1} = c$.
\item If $4 \nmid n$, $c > 1$, and $b \equiv 1 \pmod{4}$, then $h_{K/K_1} = c$.
\item If $4 \nmid n$, $c = 2$, and $b \equiv 3 \pmod{4}$, then $h_{K/K_1} = 1$
\item If $4 \nmid n$, $c > 2$, and $b \equiv 3 \pmod{4}$, then $h_{K/K_1} = c - \frac{1}{2}$.
\end{enumerate} 
\end{enumerate}
\end{theorem}

\begin{proof}
\emph{To (i):} In this case, $a$ is $2$-primitive for $K_1$, so Proposition \ref{P2conductor} shows that
$h_{K/K_1} = c+1$. 
\\
\\
\emph{To (iia):} By \cite[\S0.3]{Ep:wr}, $2^nb$ is a square in $K_1$.
\\
\\
\emph{To (iib):} This follows from Lemma \ref{L2extensions}(ii).
\\
\\
\emph{To (iic):} In this case, $K$ is contained in the compositum of $L := K_c(\sqrt[2^c]{b})$ and 
$L' := K_{c-2}(\sqrt[2^{c-2}]{2^{n/4}})$.  Now, since $b \equiv 1 \pmod{4}$, it follows from \cite[\S0.3]{Ep:wr} that
$b$ is a square in $K_1$.     
So $L = K_c(\sqrt[2^{c-1}]{b'})$, where $b'^2 = b$ and $b' \in K_1$.  Since $b' \equiv 1 \pmod 2$, it follows from Corollary
\ref{Ccruder} that $h_{L/K} \leq \max(c-1, c-1, c-1) = c-1$. 

Also, by Corollary \ref{C2specific}, $h_{L'/K_1} \leq \max(\max(0, c-3), c-1, c-2) = c-1$.
So $h_{K/K_1} \leq h_{LL'/K_1} \leq c-1$, using Lemma \ref{Lcompositum}.  But $K$ contains $K_c$, and by Proposition 
\ref{Pcyclic}, $h_{K_c/K_1} =c-1$.  So $h_{K/K_1} = c-1$.
\\
\\
\emph{To (iid):}  Consider $L$ and $L'$ as in (iic).  
Since $b$ is $2$-primitive for $K_1$, Proposition \ref{P2conductor} gives us that
$h_{L/K_1} = c$.  We have seen in (iic) that $h_{L'/K_1} \leq c-1$.
We conclude using Lemma \ref{Lcompositumexact}, applied to the subextensions $L$, $L'$, and $K$ of $LL'$.
\\
\\
\emph{To (iie):} In this case, $K$ is contained in the compositum of $L := K_c(\sqrt[2^c]{b})$ and $L' := K_{c-1}(\sqrt[2^{c-1}]{2^{n/2}})$.
As in (iic), $h_{L/K_1} = c-1$.  
Also, $2^{n/2}$ is $2$-primitive for $K_1$, so Proposition \ref{P2conductor}(ii) shows that $h_{L'/K_1} = c$.  We conclude by applying Lemma
\ref{Lcompositumexact} to the subextensions $L'$, $L$, and $K$ of $LL'$.
\\
\\ 
\emph{To (iif):} Since $-4$ is a $4$th power in $K_2$, we have that $K \cong K_2(\sqrt[4]{-b})$.  Since $-b \cong 1 \pmod{4}$, the result follows
from (iic).
\\
\\
\emph{To (iig):} Since $-4 = (1+i)^4$ (where $i^2 = -1$), it follows that $K$ is contained in the compositum of $L := K_c(\sqrt[2^c]{-2^{n-2}b})$ and 
the Galois closure $L'$ of $K_{\max(2, c-2)}(\sqrt[2^{c-2}]{1+i})$ over $K_1$.  
By (iic), $h_{L/K_1} = c-1$.  Also, since $1+i$ is $2$-primitive for $K_2$, Proposition \ref{P2conductor}(ii) 
shows that $h_{L'/K_1} = c - \frac{1}{2}$. We conclude using Lemma \ref{Lcompositumexact}, applied to the subextensions $L'$, $L$, and $K$ of $LL'$.
\end{proof}

\end{document}